\documentclass[12pt]{article}
\usepackage{amsmath}
\usepackage{amssymb}
\usepackage{amsthm}

\newtheorem{proposition}{Proposition}[section]
\newtheorem{remark}{Remark}[section]

\DeclareSymbolFont{AMSb}{U}{msb}{M}{n}
\DeclareMathSymbol{\Eout}{\mathbin}{AMSb}{"45}
\DeclareMathSymbol{\M}{\mathbin}{AMSb}{"4D}
\DeclareMathSymbol{\N}{\mathbin}{AMSb}{"4E}
\DeclareMathSymbol{\PR}{\mathbin}{AMSb}{"50}
\DeclareMathSymbol{\R}{\mathbin}{AMSb}{"52}
\DeclareMathSymbol{\SI}{\mathbin}{AMSb}{"53}
\DeclareMathSymbol{\bU}{\mathbin}{AMSb}{"55}
\DeclareMathSymbol{\bV}{\mathbin}{AMSb}{"56}
\DeclareMathSymbol{\bW}{\mathbin}{AMSb}{"57}
\DeclareMathSymbol{\bX}{\mathbin}{AMSb}{"58}
\DeclareMathSymbol{\bY}{\mathbin}{AMSb}{"59}
\DeclareMathSymbol{\bZ}{\mathbin}{AMSb}{"5A}

\title{Transfer Entropy and Directed Information in Gaussian Diffusion Processes}

\author{Nigel J.~Newton
   \thanks{School of Computer Science and Electronic Engineering, University of Essex,
   Wivenhoe Park, Colchester, CO4 3SQ, UK. ({\tt njn@essex.ac.uk})}}

\begin{document}
\maketitle

\begin{abstract}
\noindent
{\em Transfer Entropy} and {\em Directed Information} are information-theoretic measures
of the {\em directional dependency} between stochastic processes.  Following
the definitions of Schreiber and Massey in discrete time, we define and evaluate these
measures for the components of multidimensional Gaussian diffusion processes.  When the
components are jointly Markov, the Transfer Entropy and Directed Information are both
measures of {\em influence} according to a simple physical principle.  More generally,
the effect of other components has to be accounted for, and this can be achieved in more
than one way.  We propose two definitions, one of which preserves the properties of
influence of the jointly Markov case. The Transfer Entropy and Directed Information
are expressed in terms of the solutions of matrix Riccati equations, and so are easy to
compute.  The definition of continuous-time Directed Information we propose differs from
that previously appearing in the literature. We argue that the latter is not strictly directional.

Keywords: Causality, Diffusions, Directed Information, Information Flow,
Nonequilibrium Statistical Mechanics, Transfer Entropy.

2010 MSC: 60J60 60J70 62B10 82C31 93E11 94A17
\end{abstract}

\newcommand{\ca}{{\cal A}}
\newcommand{\cb}{{\cal B}}
\newcommand{\cd}{{\cal D}}
\newcommand{\cf}{{\cal F}}
\newcommand{\cg}{{\cal G}}
\newcommand{\ch}{{\cal H}}
\newcommand{\cl}{{\cal L}}
\newcommand{\cm}{{\cal M}}
\newcommand{\cn}{{\cal N}}
\newcommand{\cp}{{\cal P}}
\newcommand{\cq}{{\cal Q}}
\newcommand{\cs}{{\cal S}}
\newcommand{\cu}{{\cal U}}
\newcommand{\cv}{{\cal V}}
\newcommand{\cw}{{\cal W}}
\newcommand{\cx}{{\cal X}}
\newcommand{\cy}{{\cal Y}}
\newcommand{\cz}{{\cal Z}}

\newcommand{\bbar}{{\bar{b}}}
\newcommand{\Vbar}{{\bar{V}}}
\newcommand{\Wbar}{{\bar{W}}}
\newcommand{\Xbar}{{\bar{X}}}
\newcommand{\sigmabar}{{\bar{\sigma}}}

\newcommand{\atil}{{\tilde{a}}}
\newcommand{\btil}{{\tilde{b}}}
\newcommand{\Dtil}{{\tilde{D}}}
\newcommand{\etil}{{\tilde{e}}}
\newcommand{\ktil}{{\tilde{k}}}
\newcommand{\ltil}{{\tilde{l}}}
\newcommand{\ntil}{{\tilde{n}}}
\newcommand{\qtil}{{\tilde{q}}}
\newcommand{\Rtil}{{\tilde{R}}}
\newcommand{\Ttil}{{\tilde{T}}}
\newcommand{\util}{{\tilde{u}}}
\newcommand{\Wtil}{{\tilde{W}}}
\newcommand{\Xtil}{{\tilde{X}}}
\newcommand{\Ytil}{{\tilde{Y}}}
\newcommand{\lambdatil}{{\tilde{\lambda}}}
\newcommand{\gammatil}{{\tilde{\gamma}}}
\newcommand{\rhotil}{{\tilde{\rho}}}
\newcommand{\sigmatil}{{\tilde{\sigma}}}
\newcommand{\Xitil}{{\tilde{\Xi}}}

\newcommand{\bfU}{{\bf U}}
\newcommand{\bfV}{{\bf V}}
\newcommand{\bfW}{{\bf W}}

\newcommand{\Mhat}{\hat{M}}
\newcommand{\Xhat}{\hat{X}}
\newcommand{\Zhat}{\hat{Z}}

\newcommand{\ahac}{{\breve{a}}}
\newcommand{\Xhac}{{\breve{X}}}
\newcommand{\Yhac}{{\breve{Y}}}

\newcommand{\rank}{{\hbox{\rm rank}}}
\newcommand{\tr}{{\hbox{\rm tr}}}
\newcommand{\half}{{\frac{1}{2}}}
\newcommand{\fndot}{{\,\cdot\,}}
\newcommand{\cond}{{\,|\,}}

\section{Introduction} \label{se:intro}

{\em Transfer Entropy} and {\em Directed Information} are information-theoretic measures
of the {\em directional dependency} between stochastic processes.  They quantify the
statistical dependency between the past of one process and the future of another, and are
thus connected with notions of {\em causality} in physical systems.  Directed information
was developed in the context of Telecommunications, as a directional variant of the
{\em mutual information} between the input and output data sequences in a communication
channel \cite{mass1}.  Transfer entropy, on the other hand, was developed within the
Physical Sciences community as a means of testing for directional influence in complex
systems \cite{schr1}.  It has been applied in fields as diverse as {\em Neuroscience},
{\em Systems Biology}, {\em Climatology} and {\em Econometrics}. (See \cite{ammi1}
for a review of both quantities in a discrete-time context, their
connections with {\em Granger Causality}, and an extensive bibliography.)

This paper defines and evaluates the transfer entropy and directed information between
components of continuous-time Gaussian diffusion processes, expressing both quantities in
terms of the solutions of matrix Riccati equations.  In this setting,
both quantities lead to a single infinitesimal {\em rate} of information transfer.  Many
of the results herein carry over to diffusion processes with {\em nonlinear} dynamics.
However, such extensions make use of the theory of nonlinear filtering \cite{lish1}, and
so introduce considerable technicalities, which can all too easily obscure the meaning
of the results.
 
Information theory has its origins in telecommunications \cite{shan1}, where it provides
a fundamental limit on the rate at which data can be reliably communicated over
error-prone communication channels.  The central quantities in information theory are the
{\em mutual information} between two random variables, and its conditional variant.
{\em Conditional mutual information} is a measure of conditional dependency between a
pair of random variables, $U:\Omega\rightarrow\bfU$ and $V:\Omega\rightarrow\bfV$,
defined on a common probability space $(\Omega,\cg,\PR)$.  If $\bfU$ and $\bfV$ are
Polish spaces with Borel $\sigma$-algebras $\cu$ and $\cv$, and $\ch$ is a
sub-$\sigma$-algebra of $\cg$ then a regular $\ch$-conditional probability distribution
exists for $(U,V)$: $P_{UV|\ch}:\cu\times\cv\times\Omega\rightarrow[0,1]$.  Let
$P_{U|\ch}$ and $P_{V|\ch}$ be its marginals.  The $\ch$-conditional mutual information
is then
\begin{equation}
I(U;V|\ch) = \Eout \cd(P_{UV|\ch}\cond P_{U|\ch}\otimes P_{V|\ch}), \label{eq:mutinf}
\end{equation}
where $\cd(P\cond Q)$ is the {\em Kullback-Leibler divergence} between two probability measues,
$P$ and $Q$, defined on a common measurable space:
\begin{equation}
\cd(P\cond Q)
  = \left\{\begin{array}{ll}
           \int\frac{dP}{dQ}\log\left(\frac{dP}{dQ}\right)dQ & \text{if }P\ll Q \\
           + \infty & \text{otherwise}.
           \end{array}\right. \label{eq:kldiv}  
\end{equation}
The integral in (\ref{eq:kldiv}) is well defined since (with the convention $0\log 0=0$)
the integrand is bounded from below; the expectation in (\ref{eq:mutinf}) is well defined
since $\cd(P\cond Q)\ge 0$.  The (unconditional) mutual information $I(U;V)$ is obtained
by setting $\ch=\{\emptyset, \Omega\}$ in (\ref{eq:mutinf}).

$I(U;V|\ch)$ is non-negative; it is zero if and only if $U$ and $V$ are
$\ch$-conditionally independent.  An important property, of which we shall make frequent
use, is the {\em chain rule}; this states that
\begin{equation}
I(U;(V,W)|\ch) = I(U;V|\ch) + I(U;W|\ch\vee\sigma(V)),  \label{eq:chain}
\end{equation}
where $W:\Omega\rightarrow\bfW$ is a third random variable, and $\sigma(V)$ is the
$\sigma$-algebra generated by $V$.  The reader is referred to \cite{coth1} and
\cite{pins1} for a systematic development of such theorems of information theory;
\cite{pins1}, in particular, develops the subject at an appropriate level of abstraction.

If $\bfU$ and $\bfV$ are the path spaces of continuous-time semimartingales, then
Girsanov's theorem can be used in the computation of $I(U;V|\ch)$.  The mutual information
between a ``signal process'' $X$ and an ``observation process'' $Y$, comprising the signal
and independent additive Gaussian white noise, was first found in this way in
\cite{dunc1}.  Duncan's results were extended to the case of partial independence in
\cite{kazz1}, and to a more abstract result on Wiener space in \cite{maza2}.  The
connection between the path mutual information and the instantaneous mutual information
in this context was first investigated in \cite{maza1}.  These ideas were connected with
information and energy flows in statistical mechanical systems in \cite{mine2} and
\cite{newt3}, and with time-reversal and notions of {\em dual} linear and nonlinear
filters in \cite{newt1} and \cite{newt2}.  These papers all address problems in which
there is ``one-way influence'' between the components whose mutual information is sought,
and so lead to notions of directional {\em information flow}.

Transfer entropy and directed information were both developed for problems in which there
is ``two-way influence'' between processes.  Their definitions are based on
mutual information, and measure the dependency between the past of one process and the
future of another.  A notion of directed information between stationary discrete-time
processes appeared in the context of {\em bidirectional communication} in
\cite{mark1}.  It was refined in \cite{mass1}, and used to study communication channels
with feedback.  Massey's definition is as follows:
\begin{equation}
D_{X\rightarrow Y}^M(N)
  = \sum_{n=1}^N I((X_1, \ldots, X_n);Y_n\cond Y_1, \ldots, Y_{n-1}), \label{eq:massdi}
\end{equation}
where $(X_1,X_2,\ldots)$ is the input sequence to a communication channel, and
$(Y_1,Y_2,\ldots)$ is the corresponding output sequence.  (By definition, an empty
sequence generates the trivial $\sigma$-algebra $\{\emptyset, \Omega\}$.)  The term
``transfer entropy'' was first used in \cite{schr1}.  Schreiber's definition is as
follows:
\begin{equation}
T_{X\rightarrow Y}(k,l,n)
  = I((X_{n-l}, \ldots, X_{n-1});Y_n\cond Y_{n-k}, \ldots, Y_{n-1}); \label{eq:schrdef}
\end{equation}
it is often described as a measure of ``the disambiguation on the future of $Y$,
over and above that provided by the past of $Y$, afforded by the past of $X$''.  It is
frequently used with the common value $k=l=n-1$ for the history lengths of $X$ and $Y$,
in which case its cumulative variant,
\begin{eqnarray}
D_{X\rightarrow Y}(N)
  & = & \sum_{n=2}^N T_{X\rightarrow Y}(n-1,n-1,n) \nonumber \\
        [-1.5ex] \label{eq:symdi} \\ [-1.5ex]
  & = & \sum_{n=2}^N I((X_1, \ldots, X_{n-1});Y_n\cond Y_1, \ldots, Y_{n-1}), \nonumber
\end{eqnarray}
is similar to $D_{X\rightarrow Y}^M(N)$.

$D_{X\rightarrow Y}$ seems a more natural
definition of directed information since it does not involve the mutual information
between the values of $X$ and $Y$ at the same time instant, and so is more faithful to
the notion of directional dependency.  However, it is important to remember that Massey's
definition was made in the context of a specific application.  In \cite{mass1}, $Y_n$ is
the output of a communication channel corresponding to the input $X_n$, and so $Y_n$ actually occurs {\em after} $X_n$ in physical time (but before $X_{n+1}$), a notion that
can be made mathematically explicit by switching to a finer time scale.  Let $\Ytil_1=0$
and, for any $m\in\N$, let $\Xtil_m=X_{\lceil m/2 \rceil}$ and
$\Ytil_{m+1}=Y_{\lfloor (m+1)/2 \rfloor}$, where $\lceil \fndot \rceil$ and
$\lfloor \fndot \rfloor$ are the ceiling and floor functions, respectively.  Note that
$\Xtil$ changes only when $m$ increments from an even value, whereas $\Ytil$ changes only
when $m$ increments from an odd value, reflecting the causal relations of communication
channels with feedback.  An easy calculation shows that
\begin{equation}
D_{X\rightarrow Y}^M(N) = D_{\Xtil\rightarrow \Ytil}(2N). \label{eq:direl}
\end{equation}
(NB.~alternate terms in the sum on the right-hand side here are zero.)
The use of $D_{X\rightarrow Y}$ rather than $D_{X\rightarrow Y}^M$ also symmetrises
Massey's ``law of conservation of directed information'' \cite{mass2}, thus avoiding
shifted sequences.  For a general pair of sequences $(X_1,X_2,\ldots)$ and
$(Y_1,Y_2,\ldots)$:
\begin{eqnarray}
I((X_1,\ldots,X_N);(Y_1,\ldots,Y_N))
  & = & D_{X\rightarrow Y}(N) + D_{Y\rightarrow X}(N) \nonumber \\
        [-1.6ex] \label{eq:lawcondi} \\ [-1.6ex]
  &   & \; + \sum_{n=1}^N I(X_n;Y_n|X_1,Y_1,\ldots,X_{n-1},Y_{n-1}). \nonumber
\end{eqnarray}
The third term on the right-hand side here (called the ``instantaneous information
exchange'' in \cite{ammi1}) is zero if $X_n$ and $Y_n$ are independent, conditioned on
their pasts, which is the case for the sequences $\Xtil$ and $\Ytil$ defined above.

Since it is more faithful to the notion of directional dependency, we shall base our
definition of continuous-time directed information on $D_{X\rightarrow Y}$ rather than
$D_{X\rightarrow Y}^M$.  A definition based on $D_{X\rightarrow Y}^M$ was developed
in \cite{wekp1}, in a very general context.  However, the hypotheses used are somewhat
unripe and, like $D_{X\rightarrow Y}^M$, the information quantity obtained is not
strictly directional.

In applications, one often wants to identify
{\em causal influence} between stochastic processes.  In Neuroscience, for example, it
is useful to know which group of neurons causes which to fire.  Of course, the existence
of statistical dependency between the past of one process and the future of another is
no guarantee of causation; this can only be determined by examining the underlying
physics (and usually requires physical intervention) \cite{pear1}.  This cannot be done
in the abstract setting of this paper.  However, in order to motivate some of the
definitions we give, it is useful to have a {\em Principle of Influence} between
processes, that corresponds to causal influence in many applications.
It attributes a physical property to the time variable.
\begin{enumerate}
\item[PI:] {\em The future of one process cannot influence the past of another}.
\end{enumerate}
According to this principle, any statistical dependency between the past of one process
and the future of another arises from a combination of the influence that the former
has on the latter, and the influence that other processes have on both.  In the context of
a closed system having a filtration $(\cf(t),0\le t<\infty)$ to which all processes
are adapted, a process $X$ that is Markov with respect to $\cf$ (meaning that
$X(t)$ is $X(s)$-conditionally independent of $\cf(s)$ for all $s\le t$)
cannot be influenced by any other process.  Section \ref{se:ccgdp} of this paper defines
and evaluates the transfer entropy and directed information between two processes that
are jointly Markov in this sense;  according to PI, they measure one-way influence.
When processes are not jointly Markov, there is more than one plausible definition of
transfer entropy and directed information.  We explore marginal and conditional variants
of two such definitions in section \ref{se:threecomp}, only one of which preserves the
properties of influence of the jointly Markov case.

The following notation will be used fequently in what follows:
\begin{itemize}
\item $\M_{m,n}$ is the set of real matrices of dimension $m\times n$;
\item $\SI_n^{++}\subset\SI_n^+\subset\SI_n\subset\M_{n,n}$ are the subsets of
  positive-definite, positive-semi-definite, and symmetric $n\times n$ matrices,
  respectively;
\item $I_n\in\SI_n^{++}$ is the multiplicative identity matrix;
\item For $\mu\in\R^n$ and $v\in\SI_n^+$, $N(\mu,v)$ is the multivariate Gaussian
  distribution with mean vector $\mu$ and covariance matrix $v$. 
\end{itemize}

\section{Jointly Markov Processes} \label{se:ccgdp}

Let $(\Omega,\cg,\PR)$ be a complete probability space, on which is defined a filtration
$(\cf(t)\subset\cg,0\le t<\infty)$.  (All filtrations will be assumed to satisfy the
``usual conditions'' \cite{rowi2}.)  For some $n\ge 2$,
let $(X(t)\in\R^n,0\le t<\infty)$ be a continuous, Gaussian, $\cf$-diffusion
process defined on $\Omega$, with initial distribution $N(\mu,v)$, drift vector
$b(t)X(t)$ and diffusion matrix $a(t)$, where $\mu\in\R^n$, $v\in\SI_n^+$, and
$b:[0,\infty)\rightarrow\M_{n,n}$ and $a:[0,\infty)\rightarrow\SI_n^+$ are measurable
functions satisfying
\begin{equation}
\int_0^t \left(\|b(s)\|+\|a(s)\|\right)ds < \infty
  \quad\text{for all }0\le t<\infty. \label{eq:coefbnd}
\end{equation}
By this we mean that the distribution of $X(0)$ is $N(\mu,v)$, $X$ is adapted to $\cf$,
and the process $(M,\cf)$, defined by
\begin{equation}
M(t) = X(t) - X(0) - \int_0^t b(s)X(s)\,ds, \label{eq:Mdef}
\end{equation}
is a non-standard $\R^n$-valued Brownian motion with quadratic covariation
$[M](t)=\int_0^t a(s)\,ds$.

For some $n_1,n_2\ge 1$ with $n_1+n_2=n$, let
\begin{equation}
e_1 = \begin{bmatrix}
      I_{n_1} & 0
      \end{bmatrix}\in\M_{n_1,n}\quad\text{and}\quad
e_2 = \begin{bmatrix}
      0 & I_{n_2}
      \end{bmatrix}\in\M_{n_2,n}. \label{eq:eidef}
\end{equation}
We shall use the notation $X_i:=e_iX$, $\mu_i:=e_i\mu$, $v_{ij}:=e_ive_j^\prime$,
$b_{ij}:=e_ibe_j^\prime$ and $a_{ij}:=e_iae_j^\prime$ for the block components of $X$,
$\mu$, $v$, $b$ and $a$.  This section defines and evaluates the transfer entropy and
the directed information from $X_2$ to $X_1$.  The values of these quantities between
different components of $X$ can be found by using the methods herein on the process
$\pi X$ for an appropriately chosen non-singular matrix $\pi\in\M_{n,n}$.

\subsection{A Directed Representation for $X$} \label{se:direct}

We shall make use of the following hypotheses:
\begin{enumerate}
\item[(H1)] for any $0\le t<\infty$, $\rank(a_{11}(t))=k>0$;
\item[(H2)] for any $0\le t<\infty$, $b_{12}(t) = a_{11}(t) \gamma(t)$ 
  for some measurable function $\gamma:[0,\infty)\rightarrow\M_{n_1,n_2}$, satisfying
  \begin{equation}
  \int_0^t \|\gamma(s)^\prime (I_{n_1}+a_{11}(s))\gamma(s)\|\,ds < \infty
    \quad\text{for all }0\le t<\infty. \label{eq:gambnd}
  \end{equation}
\end{enumerate}
Let $a_{11}(t) = u(t)\lambda(t)u(t)^\prime$ be the ``reduced'' eigen-decomposition of
$a_{11}(t)$.  By this, we mean that $\lambda(t)\in\SI_k^{++}$ is a diagonal matrix
containing the {\em non-zero} eigenvalues of $a_{11}(t)$, and the columns of
$u(t)\in\M_{n_1,k}$ are the corresponding (orthonormal) eigenvectors.  Let
$\sigma(t)\in\M_{n,k+n_2}$ and $a^W(t)\in\SI_{k+n_2}^+$ be defined as follows:
\begin{equation}
\sigma(t)
  = \begin{bmatrix}
    \sigma_{11}(t) & 0 \\
    \sigma_{21}(t) & I_{n_2}
    \end{bmatrix}\quad\text{and}\quad
a^W(t)
  = \begin{bmatrix}
    I_k & 0 \\
    0   & \alpha(t)
    \end{bmatrix}, \label{eq:sigmadef}
\end{equation}
where $\sigma_{i1}(t)=a_{i1}(t)u(t)\lambda(t)^{-1/2}$ ($i=1,2$), and
\begin{equation}
\alpha(t) = a_{22}(t) - \sigma_{21}(t)\sigma_{21}(t)^\prime\in\SI_{n_2}^+;
            \label{eq:alphadef}
\end{equation}
then $a=\sigma a^W\sigma^\prime$.  Let $W$ and $\Mhat$ be defined as follows:
\begin{eqnarray}
W(t)
  & = & \begin{bmatrix} W_1(t) \\ W_2(t) \end{bmatrix}
         = \int_0^t (\sigma(s)^\prime\sigma(s))^{-1}\sigma(s)^\prime\,dM(s) \nonumber \\
        [-1.5ex] \label{eq:Wdef} \\ [-1.5ex]
\Mhat(t)
  & = & \int_0^t \sigma(s)\, dW(s), \nonumber
\end{eqnarray}
where $W_1(t)\in\R^k$ and $W_2(t)\in\R^{n_2}$.
$(W,\cf)$ is a non-standard $(k+n_2)$-vector Brownian motion with quadratic covariation
$[W](t)=\int_0^t a^W(s)ds$ and, since $\sigma(\sigma^\prime\sigma)^{-1}\sigma^\prime a=a$,
the process $(M-\Mhat,\cf)$ is a non-standard $n$-vector Brownian motion with quadratic
covariation zero.  So $M$ and $\Mhat$ are indistinguishable.

We factorise the initial covariance matrix, $v$, in a similar way. Let $l=\rank(v_{11})$,
and let $\psi\in\M_{n,l+n_2}$, $v^\Xi\in\SI_{l+n_2}^+$ and $\phi\in\SI_{n_2}^+$ be
defined as follows:
\begin{equation}
\psi  = \begin{bmatrix}
        \psi_{11} & 0 \\
        \psi_{21} & I_{n_2}
        \end{bmatrix},\quad
v^\Xi = \begin{bmatrix}
        I_l & 0 \\
        0   & \phi
        \end{bmatrix}\quad\text{and}\quad
\phi = v_{22} - \psi_{21}\psi_{21}^\prime\in\SI_{n_2}^+, \label{eq:psidef}
\end{equation}
where $\psi_{i1}=v_{i1}u_0\lambda_0^{-1/2}$ ($i=1,2$), and
$v_{11}=u_0\lambda_0 u_0^\prime$ is the reduced eigen-decomposition of $v_{11}$.
(If $l=0$ then $\psi_{11}$ and $\psi_{21}$ are void, and $v^\Xi=\phi$.)  Let
\begin{equation}
\Xi = \begin{bmatrix}
      \Xi_1 \\ \Xi_2
      \end{bmatrix}
    = (\psi^\prime\psi)^{-1}\psi^\prime(X(0)-\mu), \label{eq:Xidef}
\end{equation}
where $\Xi_1\in\R^l$ and $\Xi_2\in\R^{n_2}$.  ($\Xi_1$ is
void if $l=0$.)  Straightforward calculations show that $\Eout\|X(0)-\mu-\psi\Xi\|^2=0$,
and that $\Xi$ has the Gaussian distribution, $N(0,v^\Xi)$.

The foregoing arguments show that $X$ satisfies the following It\^{o} equation:
\begin{equation}
X(t) = \mu + \psi\Xi + \int_0^t b(s)X(s)\,ds
           + \int_0^t\sigma(s)\,dW(s). \label{eq:Xrep}
\end{equation}

\begin{remark} \label{re:canon}
In numerical implementations of the eigen-decomposition of $a_{11}$, it
may be difficult to distinguish between an eigenvalue that is zero and one that is
merely small, resulting in uncertainty about the ``noise'' dimension $k$.  Although not
mathematically necessary, the inclusion of the identity matrix in (\ref{eq:gambnd})
ensures that the transfer entropy and directed information of sections
\ref{se:transent} and \ref{se:dirinf} are not sensitive to the choice of $k$ in such
cases.  $X$ is often defined in the form (\ref{eq:Xrep}), ``up-front''.
(See, for example, \cite{sdnm1}.)
\end{remark}

In order to compute the disambiguation on the future of $X_1$ afforded by its past, it
is convenient to bring the relevant part of that past into the present.  This can be
achieved by means of a Kalman-Bucy filter for $X$ based on ``observations'' of $X_1$.
Let $\cf_1$ be the filtration generated by $X_1$, let
$(q_2(t)\in\SI_{n_2}^+,0\le t<\infty)$ satisfy the matrix Riccati equation:
\begin{eqnarray}
q_2(0)
  & = & \phi \nonumber \\
        [-1.5ex] \label{eq:riccati} \\ [-1.5ex]
\dot{q}_2
  & = & (b_{22}-a_{21}\gamma)q_2 + q_2(b_{22}-a_{21}\gamma)^\prime
        + \alpha - q_2\gamma^\prime a_{11}\gamma q_2, \nonumber
\end{eqnarray}
where $\phi$, $\gamma$ and $\alpha$ are as defined in (\ref{eq:psidef}), (H2) and
(\ref{eq:alphadef}), respectively, and let
$((\Xhat,\Wbar_1)(t)\in\R^{n+k},0\le t<\infty)$ satisfy the It\^{o} equation:
\begin{eqnarray}
\Xhat(0)
  & = & \left\{\begin{array}{ll}
        \mu & \text{if }l=0 \\
        \mu + \psi\begin{bmatrix} I_l & 0\end{bmatrix}^\prime\Xi_1 &  \text{otherwise}
        \end{array}\right.,\quad \Wbar_1(0)=0 \nonumber \\
d\Xhat
  & = & b\Xhat dt + \left(\sigma\begin{bmatrix} I_k & 0\end{bmatrix}^\prime
         + e_2^\prime q_2\gamma^\prime\sigma_{11}\right)d\Wbar_1\quad
         \label{eq:filter} \\
d\Wbar_1
  & = & \sigma_{11}^\prime\gamma e_2(X-\Xhat)dt + dW_1. \nonumber
\end{eqnarray}

\begin{proposition} \label{pr:filters}
If (H1) and (H2) hold, then:
\begin{enumerate}
\item[(i)] for any $0\le t<\infty$, the Gaussian distribution
$N(\Xhat(t),e_2^\prime q_2(t)e_2)$ is a regular $\cf_1(t)$-conditional distribution for
$X(t)$;
\item[(ii)] $(\Wbar_1,\cf_1)$ is an $\R^k$-valued standard Brownian motion;
\item[(iii)] for any $0\le s\le t<\infty$,
$\cf_1(t)=\cf_1(s)\vee\sigma(\Wbar_1(r)-\Wbar_1(s),\,s\le r\le t)$.
\end{enumerate}
\end{proposition}

\begin{proof}
We begin by proving part (i) in the special case that $t=0$.  If $l=0$ then $X_1(0)$
is non-random, and the (unconditional) distribution, $N(\mu,e_2^\prime\phi e_2)$, is also
a regular $X_1(0)$-conditional distribution.  On the other hand, if $l>0$ then
$\Xi_1 = (\psi_{11}^\prime\psi_{11})^{-1}\psi_{11}^\prime(X_1(0)-\mu_1)$ is
$\cf_1(0)$-measurable, and $\Xi_2$ is independent of $\cf_1(0)$.  That
$N(\Xhat_2(0),\phi)$ is a regular $X_1(0)$-conditional distribution for $X_2(0)$
follows from the fact that $X_2(0)=\mu_2+\psi_{21}\Xi_1+\Xi_2$.

Let $p=e_2^\prime q_2e_2$; since $b_{12}=a_{11}\gamma$,
\begin{eqnarray*}
\dot{p} 
  & = & (b-\rho\sigma_{11}^\prime\gamma e_2)p
        + p(b-\rho\sigma_{11}^\prime\gamma e_2)^\prime + a - \rho\rho^\prime
          - p e_2^\prime\gamma^\prime a_{11}\gamma e_2p \\
  & = & bp + pb^\prime + a
        - \big(\rho+p e_2^\prime\gamma^\prime\sigma_{11}\big)
        \big(\rho+p e_2^\prime\gamma^\prime\sigma_{11}\big)^\prime,
\end{eqnarray*}
where $\rho=\sigma\begin{bmatrix} I_k & 0\end{bmatrix}^\prime$.
Let $Y$ be the following $k$-vector ``observations'' process,
\[
Y(t) = \int_0^t \sigma_{11}(s)^\prime\gamma(s)e_2X(s)\,ds + W_1(t),
\]
and let $\cf^Y$ be the filtration it generates.  Now $e_2^\prime q_2=p e_2^\prime$,
and so the diffusion coefficient in the second equation in (\ref{eq:filter}) is
$\rho+p e_2^\prime\gamma^\prime\sigma_{11}$. It is a standard result of Kalman-Bucy
filtering with correlated noise processes (see, for example, Theorem 10.3 in \cite{lish1})
that $N(\Xhat(t),p(t))$ is a regular $\cf_1(0)\vee\cf^Y(t)$-conditional distribution
for $X(t)$.  Since $\Xhat_1=X_1$, $\cf_1(t)\subset\cf_1(0)\vee\cf^Y(t)$ for all $t$.
It follows from (H2) that
\[
Y(t) = \int_0^t \big(\sigma_{11}(s)^\prime\sigma_{11}(s)\big)^{-1}\sigma_{11}(s)^\prime 
           (dX_1(s)-b_{11}(s)X_1(s)ds),
\]
and since $(X_1,\cf_1)$ is a semimartingale (Stricker's theorem),
$\cf^Y(t)\subset\cf_1(t)$.  This completes the proof of part (i).
 
The fact that $(\Wbar_1,\cf_1)$ is a martingale with quadratic covariation
$[\Wbar_1](t)=I_kt$ is easily established, and this proves part (ii).  Part (iii) follows
since the second equation in (\ref{eq:filter}) has a strong solution.
\end{proof}

\subsection{The Transfer Entropy} \label{se:transent}

For any continuous, vector-valued process $(\Theta(t)\in\R^m,\,0\le t<\infty)$ and any
$0\le s< t<\infty$, we use the notation $\Theta(s,t)$ for the $C([s,t];\R^m)$-valued
random variable $(\Theta(r),\,s\le r\le t)$.  (The space $C([s,t];\R^m)$ is assumed to
be metrised by the maximum norm.)  We define the transfer entropy from $X_2$ to $X_1$
as follows:  
\begin{equation}
T_{2\rightarrow 1}(s,t)
  = I(X_2(0,s);X_1(s,t)\cond\cf_1(s))\quad\text{for }0\le s\le t<\infty. \label{eq:transent}
\end{equation}

\begin{proposition} \label{pr:transent}
If (H1) and (H2) hold then, for any $0\le s\le t <\infty$,
\begin{equation}
T_{2\rightarrow 1}(s,t)
   = \half\int_s^t\tr\big(\gamma(r)^\prime a_{11}(r)\gamma(r)(q_2(r)-\qtil_2(r))\big)\,dr,
     \label{eq:T21val}
\end{equation}
where $\gamma$ and $q_2$ are as defined in (H2) and (\ref{eq:riccati}), respectively,
$(\qtil_2(r),\,s\le r<\infty)$ satisfies (\ref{eq:riccati}) over the time interval
$[s,\infty)$ and $\qtil_2(s)=0$.
\end{proposition}

\begin{proof}
Fix $0\le s<\infty$, let $(Z(t)\in\R^{n+n_2},\, s\le t<\infty)$ be the process with
components,
\[
Z_1(t) = \begin{bmatrix}
         X_1(t) \\ X_2(s)
         \end{bmatrix}\quad\text{and}\quad
Z_2(t) = X_2(t),
\]
and let $(\cf_1^Z(t),s\le t<\infty)$ be the filtration generated by $Z_1$.  Then $Z$
satisfies the It\^{o} equation,
\[
Z(t) = Z(s) + \int_s^t e^\prime b(r)e Z(r)dr + \int_s^t e^\prime\sigma(r)dW(r),
\]
where $e=\begin{bmatrix}e_1^\prime & 0 & e_2^\prime\end{bmatrix}\in\M_{n,n+n_2}$. Let
$((\Zhat,\Wbar_1^Z)(t)\in\R^{n+n_2+k},\ s\le t<\infty)$ satisfy the It\^{o} equation
\begin{eqnarray*}
\Zhat(s)
  & = & Z(s),\quad \Wbar_1^Z(s) = 0 \\
d\Zhat
  & = & e^\prime be\Zhat dt
        + e^\prime\left(\rho+e_2^\prime\qtil_2\gamma^\prime\sigma_{11}\right)d\Wbar_1^Z, \\
d\Wbar_1^Z
  & = & \sigma_{11}^\prime\gamma e_2e(Z-\Zhat)dt + dW_1.
\end{eqnarray*}
It is easily verified that $Z$ satisfies the hypotheses of Proposition \ref{pr:filters},
with $s$ playing the role of time 0, $n$ playing the role of $n_1$, and $n+n_2$ playing
the role of $n$. Proposition \ref{pr:filters} shows that
$N(\Zhat(t),e^\prime e_2^\prime\qtil_2(t)e_2e)$ is a regular $\cf_1^Z(t)$-conditional
distribution for $Z(t)$, and that $(\Wbar_1^Z,\cf_1^Z)$ is a $k$-dimensional standard
Brownian motion.  (In particular, since $X$ is Markov with respect to $\cf$, $\Wbar_1^Z$
is independent of $\cf(s)$.)

Let $\bbar=b_{22}-(\sigma_{21}+\qtil_2 \gamma^\prime\sigma_{11})\sigma_{11}^\prime\gamma$
and $\sigmabar=(\qtil_2-q_2)\gamma^\prime\sigma_{11}$. It follows from (\ref{eq:coefbnd})
and (\ref{eq:gambnd}), and the continuity of $q_2$ and $\qtil_2$, that
$\int_0^t(\|\bbar(s)\|+\|\sigmabar(s)\|^2)ds<\infty$ for all $0\le t<\infty$.  Let
$\Theta(t)=\Zhat_2(t)-\Xhat_2(t)$ and $\Wbar_1^+(t)=\Wbar_1(t)-\Wbar_1(s)$; then
$(\Theta,\Wbar_1^+)$ satisfies the equations
\begin{eqnarray}
\Theta(s) & = & X_2(s)-\Xhat_2(s),\quad \Wbar_1^+(s) = 0 \nonumber \\
d\Theta
  & = & \bbar\Theta dt + \sigmabar d\Wbar_1^+, \label{eq:Thetadef} \\
d\Wbar_1^+
  & = & \sigma_{11}^\prime\gamma\Theta dt + d\Wbar_1^Z,
        \quad\, . \nonumber
\end{eqnarray}
It is easily verified that $\Eout(\Theta(t)\cond\cf_1(t))=0$, and
$\Eout(\Theta(t)\Theta(t)^\prime\cond\cf_1(t))=q_2(t)-\qtil_2(t)$.  Furthermore,
(\ref{eq:Thetadef}) has a strong solution, and so Theorem 3.1 in \cite{maza2} shows that
\begin{equation}
I(X(0,s);\Wbar_1^+(s,t))
  = \half\int_s^t\tr\left(\gamma(r)^\prime a_{11}(r)\gamma(r)(q_2(r)-\qtil_2(r))\right)dr.
    \label{eq:I21val}
\end{equation}
(See also \cite{kazz1} for the case in which $k=1$.)  Now
\begin{eqnarray}
I(X(0,s);\Wbar_1^+(s,t))
  & = & I(X(0,s);\Wbar_1^+(s,t)) - I(X_1(0,s);\Wbar_1^+(s,t)) \nonumber \\
  & = & I(X_2(0,s);\Wbar_1^+(s,t)\cond\cf_1(s)) \nonumber \\
        [-1.5ex] \label{eq:xwbplus} \\ [-1.5ex]
  & = & I(X_2(0,s);(X_1(s),\Wbar_1^+(s,t))\cond\cf_1(s)) \nonumber \\
  & = & T_{2\rightarrow 1}(s,t), \nonumber
\end{eqnarray}
where we have used Proposition \ref{pr:filters}(ii) and the fact that the mutual
information between independent random variables is zero in the first step, the
chain rule in the second step, and Proposition \ref{pr:filters}(iii) and the invariance
of conditional mutual information under measurable isomorphisms in the final step.  The
statement of the proposition now follows from (\ref{eq:I21val}) and (\ref{eq:xwbplus}).
\end{proof}

\begin{remark} \label{eq:notmark}
Since $X$ is Markov, $T_{2\rightarrow 1}(s,t)=I(X_2(s);X_1(s,t)|\cf_1(s))$,
and so the history of $X_2$ is unimportant in (\ref{eq:transent}).
The transfer entropy for different histories of $X_1$ can be obtained by re-defining
the time origin at which the Riccati equation (\ref{eq:riccati}) is initialised.
\end{remark}

Substituting $W_1$ from the third equation of (\ref{eq:filter}) into (\ref{eq:Xrep}),
we can express $X_2$ in the form $X_2=X_2^1+X_2^2$, where
\begin{eqnarray}
X_2^1(0) & = & \mu_2 + \psi_{21}\Xi_1,\quad X_2^2(0) = \Xi_2 \nonumber\\
dX_2^1
  & = & (b_{22}-a_{21}\gamma)X_2^1dt + b_{21}X_1dt + a_{21}\gamma\Xhat_2dt
        + \sigma_{21}d\Wbar_1  \label{eq:x2comp} \\
dX_2^2
  & = & (b_{22}-a_{21}\gamma)X_2^2dt + dW_2, \nonumber
\end{eqnarray}
and $\psi_{21}\Xi_1$ is void if $l=0$.
Since these equations have strong solutions, the components $X_2^1$ and $X_2^2$ are
adapted to the filtrations generated by $X_1$ and $(\Xi_2,W_2)$, respectively.  (This is
perhaps surprising since, according to (\ref{eq:Xrep}), $W_1$ ``drives'' $X_2$ and is
typically not adapted to $\cf_1$.)  It follows that
\begin{equation}
T_{2\rightarrow 1}(s,t) = I((\Xi_2,W_2(0,s));X_1(s,t)\cond\cf_1(s)); \label{eq:transour}
\end{equation}
in particular, no part of $T_{2\rightarrow 1}$ has its origins in the fact that $X_2$
shares the noise component $W_1$ with $X_1$ if $a_{12}\neq 0$.

\subsection{The Directed Information} \label{se:dirinf}

Let $D_{2\rightarrow 1}(t)$ be the following {\em directed information}:
\begin{equation}
D_{2\rightarrow 1}(t) = \int_0^t R_{2\rightarrow 1}(s)\,ds,\quad\text{where}\quad
R_{2\rightarrow 1}(t)
  = \lim_{\delta\downarrow 0}\delta^{-1}T_{2\rightarrow 1}(t,t+\delta). \label{eq:r21val}
\end{equation}
It is natural to think of $R_{2\rightarrow 1}$ as representing a {\em flow} of Shannon
information from $X_2$ to $X_1$.  The following proposition develops this idea, showing
that it is the rate at which $X_1$ learns about the process $(\Xi_2,W_2)$. 
  
\begin{proposition} \label{pr:mutinf}
If (H1) and (H2) hold then, for any $0\le t<\infty$,
\begin{eqnarray}
D_{2\rightarrow 1}(t)
  & = & I((\Xi_2, W_2(0,t)); X_1(0,t)) \nonumber \\
        [-1.5ex] \label{eq:intr21} \\ [-1.5ex]
  & = & \half\int_0^t \tr\big(\gamma(s)^\prime a_{11}(s)\gamma(s)q_2(s)\big)\,ds.
        \nonumber
\end{eqnarray}
\end{proposition}

\begin{proof}
Let $\bbar=b_{22}-a_{21}\gamma$ and $\sigmabar=-q_2\gamma^\prime\sigma_{11}$. It
follows from (\ref{eq:coefbnd}) and (\ref{eq:gambnd}), and the continuity of $q_2$,
that $\int_0^t(\|\bbar(s)\|+\|\sigmabar(s)\|^2)ds < \infty$ for all $0\le t<\infty$.
Let $\Phi(t)=X_2(t)-\Xhat_2(t)$; then $(\Phi,\Wbar_1)$ satisfies the equations
\begin{eqnarray}
\Phi(0) & = & \Xi_2,\quad \Wbar_1(0) = 0 \nonumber \\
d\Phi
  & = & \bbar\Phi dt + \sigmabar d\Wbar_1 + dW_2 \label{eq:Phidef} \\
d\Wbar_1
  & = & \sigma_{11}^\prime\gamma\Phi dt + dW_1. \nonumber
\end{eqnarray}
Now (\ref{eq:Phidef}) has a strong solution, and so $(\Xi_2,W_2(0,t),\Wbar_1(0,t))$ is
independent of $\Xi_1$.  It thus follows from Proposition \ref{pr:filters}(iii) and
the chain rule (\ref{eq:chain}) that
\begin{eqnarray*}
I((\Xi_2,W_2(0,t));X_1(0,t))
  & = & I((\Xi_2,W_2(0,t));\Wbar_1(0,t)) \\
  & = & \half\int_0^t\tr(\gamma(s)^\prime a_{11}(s)\gamma(s)q_2(s))\,ds,
\end{eqnarray*}
where we have used Theorem 3.1 of \cite{maza2} in the second step.
\end{proof}

There is an important difference between $D_{2\rightarrow 1}$ and the obvious
continuous-time extension of Massey's definition (\ref{eq:massdi}).  From the latter
we might propose the definition
$D_{2\rightarrow 1}^M(t)=\int_0^t R_{2\rightarrow 1}^M(s)ds$, where
\begin{equation}
R_{2\rightarrow 1}^M(t)
  = \lim_{\delta\downarrow 0}\delta^{-1}I(X_2(0,t+\delta);X_1(t,t+\delta)\cond\cf_1(t));
    \label{eq:tildef}
\end{equation}
cf.~(32) in \cite{wekp1}.  However, as pointed out in the introduction, this is truly
directional only if there is a physical time delay between $X_2(t)$ and $X_1(t)$.  This
is not so here since $X_1(t)$ can influence $X_2(t+\delta)$, through the coefficient
$b_{21}$, for arbitrarily small $\delta>0$.  The mutual information in (\ref{eq:tildef})
can be infinite since $X_1(t,t+\delta)$ and $X_2(t,t+\delta)$ typically have
non-zero quadratic covariation, $[X_1,X_2](t)=\int_0^t a_{12}(s)ds$, which results in
the singularity of their joint distribution with respect to its product of marginals.
If $X_1$ and $X_2$ share no noise ($a_{12}=0$) then $R_{2\rightarrow 1}^M$ coincides
with $R_{2\rightarrow 1}$.

\section{Processes that are not Jointly Markov} \label{se:threecomp}

Let $X$ be as defined in section \ref{se:ccgdp}, and satisfy (H1) and (H2).  In this
section we suppose that $n_2\ge 2$, and sub-divide $X_2$ into two components.  For
some $\ntil_2,\ntil_3\ge 1$ with $\ntil_2+\ntil_3=n_2$, let
\begin{equation}
\etil_1 = e_1, \quad
\etil_2 = \begin{bmatrix} I_{\ntil_2} & 0\end{bmatrix}e_2\in\M_{\ntil_2,n}
          \quad\text{and}\quad
\etil_3 = \begin{bmatrix} 0 & I_{\ntil_3}\end{bmatrix}e_2\in\M_{\ntil_3,n},
          \label{eq:eitdef}
\end{equation}
where $e_i$ is as defined in (\ref{eq:eidef}).  We shall use the notation
$\Xtil_i:=\etil_iX$, $\btil_{ij}:=\etil_i b\etil_j^\prime$ and
$\atil_{ij}:=\etil_i a\etil_j^\prime$ ($i,j=1,2,3$) for the block components of $X$, $b$
and $a$.  Proposition \ref{pr:transent} evaluates the transfer entropy from
$(\Xtil_2,\Xtil_3)$ to $\Xtil_1$ ($\Ttil_{(2,3)\rightarrow 1}=T_{2\rightarrow 1}$).
In what follows, we split $\Ttil_{(2,3)\rightarrow 1}$ into separate components
from $\Xtil_2$ and $\Xtil_3$, in two different ways.

\subsection{Splitting by components of $X_2$}

According to the chain rule (\ref{eq:chain}):
\begin{equation}
\Ttil_{(2,3)\rightarrow 1}(s,t) =
T_{2\rightarrow 1}(s,t)
  = \Ttil_{2\rightarrow 1}^X(s,t) + \Ttil_{3\rightarrow 1|2}^X(s,t), \label{eq:Xsplit}
\end{equation}
where
\begin{eqnarray}
\Ttil_{2\rightarrow 1}^X(s,t)
  & = & I(\Xtil_2(0,s);\Xtil_1(s,t)\cond\cf_1(s)), \nonumber \\
        [-1.5ex] \label{eq:Xtecomp} \\ [-1.5ex]
\Ttil_{3\rightarrow 1|2}^X(s,t)
  & = & I(\Xtil_3(0,s);\Xtil_1(s,t)\cond\cf_{12}(s)), \nonumber  
\end{eqnarray}
and $\cf_{12}$ is the filtration generated by $(\Xtil_1,\Xtil_2)$.

Let $\alpha_{ij}:=\etil_ie_2^\prime\alpha e_2\etil_j^\prime$ ($i,j=2,3$), where $\alpha$
is as defined in (\ref{eq:alphadef}).  In order to compute $\Ttil_{2\rightarrow 1}^X$ and
$\Ttil_{3\rightarrow 1|2}^X$, we first factorise $\alpha$ as we did $a$ in section
\ref{se:direct}, making use of the following hypotheses:
\begin{enumerate}
\item[(H3)] for any $0\le t<\infty$, $\rank(\alpha_{22}(t))=\ktil>0$;
\item[(H4)] for any $0\le t<\infty$,
  $\btil_{23}(t) = \atil_{21}(t)\gamma(t)e_2\etil_3^\prime + \alpha_{22}(t)c(t)$, where
  $\gamma$ is as defined in (H2), and $c:[0,\infty)\rightarrow\M_{\ntil_2,\ntil_3}$, is
  a measurable function satisfying
  \begin{equation}
  \int_0^t \|c(s)^\prime (I_{\ntil_2}+\alpha_{22}(s))c(s)\|\,ds < \infty
    \quad\text{for all }0\le t<\infty. \label{eq:gam23bnd}
  \end{equation}
\end{enumerate}
Let $\alpha_{22}(t)=\util(t)\lambdatil(t)\util(t)^\prime$ be the reduced
eigen-decomposition of $\alpha_{22}(t)$; then
$a(t)=\tau(t)a^V(t)\tau(t)^\prime$, where
\begin{equation}
\tau(t)
  = \begin{bmatrix}
     \tau_{11}(t) &       0      &   0    \\
     \tau_{21}(t) & \tau_{22}(t) &   0    \\
     \tau_{31}(t) & \tau_{32}(t) & I_{n_3}
     \end{bmatrix},\quad
a^V(t)
  = \begin{bmatrix}
     I_{k+\ktil} &   0 \\
          0      & \beta(t)
     \end{bmatrix}, \label{eq:tauav}
\end{equation}
$\tau_{11}(t)=\sigma_{11}(t)$, $\tau_{i1}(t)=\etil_ie_2^\prime\sigma_{21}(t)$ and
$\tau_{i2}(t)=\alpha_{i2}(t)\util(t)\lambdatil(t)^{-1/2}$ ($i=2,3$), and
\begin{equation}
\beta(t) = \alpha_{33}(t) - \tau_{32}(t)\tau_{32}(t)^\prime. \label{eq:betdef}
\end{equation}
Similarly, let $\phi_{ij}:=\etil_ie_2^\prime\phi e_2\etil_j^\prime$ ($i,j=2,3$), where
$\phi$ is as defined in (\ref{eq:psidef}), and let $\ltil=\rank(\phi_{22})$.
If $\ltil>0$ then $\phi_{22}$ admits a reduced eigen-decomposition
$\phi_{22}=\util_0\lambdatil_0\util_0^\prime$, and $v=\eta v^\Theta\eta^\prime$,
where
\begin{equation}
\eta = \begin{bmatrix}
        \eta_{11} &     0     &   0    \\
        \eta_{21} & \eta_{22} &   0    \\
        \eta_{31} & \eta_{32} & I_{n_3}
        \end{bmatrix},\quad
v^\Theta = \begin{bmatrix}
           I_{l+\ltil} &   0 \\
              0        & \theta
           \end{bmatrix},\quad
\theta = \phi_{33} - \eta_{32}\eta_{32}^\prime,  \label{eq:thetdef}
\end{equation}
$\eta_{11}=\psi_{11}$, and $\eta_{i1}=\etil_ie_2^\prime\psi_{21}$ and $\eta_{i2}=\phi_{i2}\util_0\lambdatil_0^{-1/2}$ ($i=2,3$).
($\eta_{i1}$ is void if $l=0$, $\eta_{i2}$ is void if $\ltil=0$, and
$v^\Theta=\theta$ if $l+\ltil=0$.)

The foregoing arguments show that $X$ satisfies the It\^{o} equation:
\begin{equation}
X(t) = \mu + \eta\Theta + \int_0^t b(s)X(s)\,ds + \int_0^t\tau(s)\,dV(s), \label{eq:X3rep}
\end{equation}
where
\begin{equation}
\Theta = (\eta^\prime\eta)^{-1}\eta^\prime(X(0)-\mu),\quad
V(t) = \int_0^t(\tau(s)^\prime\tau(s))^{-1}\tau(s)^\prime dM(s), \label{eq:ThetV}
\end{equation}
and $M$ is as defined in (\ref{eq:Mdef}).
$\Theta$ is an $\cf(0)$-measurable $(l+\ltil+\ntil_3)$-vector Gaussian random variable
with mean zero and covariance matrix $v^\Theta$, and $(V,\cf)$ is a
$(k+\ktil+\ntil_3)$-vector, non-standard Brownian motion with quadratic covariation
$[V](t)=\int_0^t a^V(s)\,ds$.  Let
\begin{equation}
\etil_{(12)} = \begin{bmatrix} I_{n_1+\ntil_2} & 0\end{bmatrix}\in\M_{n_1+\ntil_2,n}.
\end{equation}
We shall use the notation $\Xtil_i:=\etil_iX$, $\btil_{ij}:=\etil_ib\etil_j^\prime$ and
$\atil_{ij}:=\etil_ia\etil_j^\prime$ ($i,j=1,2,3,(12)$) for the block components of $X$,
$b$ and $a$.

Let $(q_3(t)\in\SI_{\ntil_3}^+,0\le t<\infty)$ satisfy the following matrix Riccati
equation:
\begin{eqnarray}
q_3(0)
  & = & \theta \nonumber \\
        [-1.5ex] \label{eq:q3def} \\ [-1.5ex]
\dot{q}_3
  & = & (\btil_{33}-\atil_{3(12)}\gammatil)q_3
        + q_3(\btil_{33}-\atil_{3(12)}\gammatil)^\prime
        + \beta - q_3\gammatil^\prime \atil_{(12)(12)}\gammatil q_3,\qquad \nonumber
\end{eqnarray}
where $\theta$ and $\beta$ are as defined in (\ref{eq:thetdef}) and (\ref{eq:betdef}), and
\begin{equation}
\gammatil(t)
  = \begin{bmatrix}
    \gamma(t)e_2\etil_3^\prime-\tau_{11}(t)(\tau_{11}(t)^\prime\tau_{11}(t))^{-1}
    \tau_{21}(t)^\prime c(t) \\
    c(t)
    \end{bmatrix} \in \M_{n_1+\ntil_2,\ntil_3}. \label{eq:gamtil}
\end{equation}

\begin{proposition} \label{pr:transent3}
If (H1)--(H4) hold, then, for any $0\le s\le t<\infty$,
\begin{equation}
\Ttil_{3\rightarrow 1|2}^X(s,t)
  = \half\int_s^t\tr\big(\etil_3e_2^\prime\gamma(r)^\prime
    a_{11}(r)\gamma(r)e_2\etil_3^\prime
    (q_3(r)-\qtil_3(r))\big)dr, \label{eq:TX31val}
\end{equation}
where $(\qtil_3(t), s\le t<\infty)$ satisfies (\ref{eq:q3def}) over the time interval
$[s,\infty)$, and $\qtil_3(s)=0$.  (NB.~$\Ttil_{2\rightarrow 1}^X$ can be found from
(\ref{eq:T21val}), (\ref{eq:Xsplit}) and (\ref{eq:TX31val}).)
\end{proposition}

\begin{proof}
It follows from (H2), (H4) and (\ref{eq:gamtil}) that
$\btil_{(12)3}=\atil_{(12)(12)}\gammatil$, and so the process $X$, regarded as comprising
the two components $\Xtil_{(12)}$ and $\Xtil_3$, satisfies hypotheses (H1) and (H2),
with $\gammatil$ playing the role of $\gamma$.  Proposition \ref{pr:filters} thus shows
that the Gaussian distribution $N(\Xbar(t),\etil_3^\prime q_3(t)\etil_3)$ is a regular
$\cf_{12}(t)$-conditional distribution for $X(t)$, where
$((\Xbar,\Vbar_{(12)})(t)\in\R^{n+k+\ktil},0\le t<\infty)$ satisfies the It\^{o} equation
\begin{eqnarray}
\Xbar(0)
  & = & \left\{\begin{array}{ll}
               \mu & \text{if }l+\ltil=0 \\
               \mu + \eta\begin{bmatrix}I_{l+\ltil} & 0\end{bmatrix}^\prime
               \Theta_{(12)} &  \text{otherwise}
			     \end{array}\right.,\quad \Vbar_{(12)}(0) = 0\quad\nonumber \\
d\Xbar
  & = & b\Xbar dt
        + \left(I_n
          +\etil_3^\prime q_3\gammatil^\prime \etil_{(12)}\right)\rhotil\;d\Vbar_{(12)}
         \label{eq:filter3} \\
d\Vbar_{(12)}
  & = & \rhotil^\prime \etil_{(12)}^\prime\gammatil e_3(X-\Xhat) dt
        + dV_{(12)}. \nonumber
\end{eqnarray}
Here $\Theta_{(12)}=\begin{bmatrix}I_{l+\ltil} & 0\end{bmatrix}\Theta$,
$V_{(12)}=\begin{bmatrix}I_{k+\ktil} & 0\end{bmatrix}V$ and
$\rhotil=\tau\begin{bmatrix}I_{k+\ktil} & 0\end{bmatrix}^\prime$.  ($\Theta_{(12)}$ is
void if $l+\ltil=0$.)

The process $\Xbar$, regarded as comprising the two components $\Xtil_1$ and
$(\Xtil_2, \Xbar_3)$, itself satisfies hypotheses (H1) and (H2), and so we may apply
Propositions \ref{pr:filters} and \ref{pr:transent} to find the transfer entropy
from $(\Xtil_2, \Xbar_3)$ to $\Xtil_1$.  Straightforward calculations show that the
$\cf_1(t)$-conditional covariance matrix of $(\Xtil_2(t),\Xbar_3(t))$ is
$q_2(t)-e_2\etil_3^\prime q_3(t)\etil_3e_2^\prime$, where $q_2$ is as defined in
(\ref{eq:riccati}), and so, according to (\ref{eq:T21val}), the transfer entropy from
$(\Xtil_2,\Xbar_3)$ to $\Xtil_1$ is
\[
T_{2\rightarrow 1}(s,t)
    - \half\int_s^t \tr\big(\gamma(r)^\prime a_{11}(r)\gamma(r)
      e_2\etil_3^\prime (q_3(r)-\qtil_3(r))\etil_3e_2^\prime\big)dr.
\]
Since $\Xbar_3$ is adapted to $\cf_{12}$, the chain rule shows that
\[
I((\Xtil_2,\Xbar_3)(0,s);\Xtil_1(s,t)\cond\cf_1(s))
  = I(\Xtil_2(0,s);\Xtil_1(s,t)\cond\cf_1(s)) = \Ttil_{2\rightarrow 1}^X(s,t).
\]
Together with (\ref{eq:Xsplit}), this proves (\ref{eq:TX31val}).
\end{proof}

$\Ttil_{2\rightarrow 1}^X$ and $\Ttil_{3\rightarrow 1|2}^X$ are natural candidates for
the transfer entropy between components of the three-component process $X$.  They both
quantify the disambiguation on the future of $\Xtil_1$ afforded by the past of
another component.  However, unlike $T_{2\rightarrow 1}$, $\Ttil_{2\rightarrow 1}^X$
is typically not a measure of the {\em influence} that $\Xtil_2$ has on $\Xtil_1$, since
it can be strictly positive even if $\btil_{12}=0$ and $\btil_{32}=0$.  This is so if
$\Xtil_1$ and $\Xtil_2$ inherit a common history from $\Xtil_3$.  Correspondingly,
$\Ttil_{3\rightarrow 1|2}^X$, measures only the {\em direct} influence that $\Xtil_3$
has on $\Xtil_1$; it does not include any influence that it may have on $\Xtil_1$
{\em via} $\Xtil_2$.

Applying Proposition \ref{pr:mutinf} to the process $\Xbar$ of (\ref{eq:filter3}), we
obtain the directed information associated with $\Ttil_{2\rightarrow 1}^X$:
\begin{equation}
\Dtil_{2\rightarrow 1}^X(t)
  = \int_0^t \Rtil_{2\rightarrow 1}^X(s)ds
  = I((\Theta_2,\Vbar_2(0,s));\Xtil_1(0,t)),
\end{equation}
where
\[
\Rtil_{2\rightarrow 1}^X(t)
  = \lim_{\delta\downarrow 0}\delta^{-1}\Ttil_{2\rightarrow 1}^X(t,t+\delta),
\]
$\Theta_2=\begin{bmatrix} 0 & I_\ltil \end{bmatrix}\Theta_{(12)}$ and
$\Vbar_2=\begin{bmatrix} 0 & I_\ktil \end{bmatrix}\Vbar$. $\Ttil_{2\rightarrow 1}^X$
and $\Dtil_{2\rightarrow 1}^X$ are measures of the influence that $\Xtil_2$ has on
$\Xtil_1$ in the system (\ref{eq:filter3}), where the non-$\cf_{12}$-adapted component
$\Xtil_3$ is replaced by additional dynamics that preserve the joint distribution of
$\Xtil_1$ and $\Xtil_2$.  $(\Theta_2,\Vbar_2)$ is Markov with respect to $\cf_{12}$, but
not with respect to $\cf$, illustrating the importance of the filtration in determining
influence according to PI. We can clearly also define
$\Dtil_{3\rightarrow 1|2}^X = D_{2\rightarrow 1}-\Dtil_{2\rightarrow 1}^X$.

\subsection{Splitting by components of $(\Xi_2,W_2)$}

Unlike $X_1$ and $X_2$, all components of $(\Xi,W)$ are Markov with respect to $\cf$, and
so, according to PI, any dependency existing between the past of a component
of $(\Xi,W)$ and the future of $X$ can be attributed to the influence that the former
has on the latter.  In this section we apply the chain rule to the right-hand side of
(\ref{eq:transour}) in order to identify the influence that sub-components of
$(\Xi_2,W_2)$ have on $\Xtil_1$.  Although this can be done without further assumptions,
the split obtained is more directly connected with the components $\Xtil_2$ and $\Xtil_3$
under the following hypothesis:
\begin{enumerate}
\item[(H5)] $\phi_{23}=0$, and for any $0\le t<\infty$, $\alpha_{23}(t)=0$.
\end{enumerate}
According to (\ref{eq:transour}) and the chain rule
\begin{equation}
\Ttil_{(2,3)\rightarrow 1}(s,t) =
T_{2\rightarrow 1}(s,t)
  = \Ttil_{2\rightarrow 1}^W(s,t) + \Ttil_{3\rightarrow 1|2}^W(s,t), \label{eq:Wsplit}
\end{equation}
where
\begin{eqnarray}
\Ttil_{2\rightarrow 1}^W(s,t)
  & = & I((\Xitil_2,\Wtil_2(0,s));\Xtil_1(s,t)\cond\cf_1(s)), \nonumber \\
        [-1.5ex] \label{eq:Xtesplit} \\ [-1.5ex]
\Ttil_{3\rightarrow 1|2}^W(s,t)
  & = & I((\Xitil_3,\Wtil_3(0,s));\Xtil_1(s,t)\cond\cf_1(s)\vee\cf_2^W(s)), \nonumber
\end{eqnarray}
$\Xitil_i=\etil_ie_2^\prime\Xi_2$ and $\Wtil_i=\etil_ie_2^\prime W_2$ ($i=2,3$), and
$\cf_2^W$ is the filtration generated by $(\Xitil_2,\Wtil_2)$.  If (H5) holds then
$(\Xitil_2,\Wtil_2)$ and $(\Xitil_3,\Wtil_3)$ influence $X$ only through the components
$\Xtil_2$ and $\Xtil_3$, respectively, and so it is reasonable to think of
$\Ttil_{2\rightarrow 1}^W$ and $\Ttil_{3\rightarrow 1|2}^W$ as being transfer entropies
from $\Xtil_2$ and $\Xtil_3$.

Let $(q_2^c(t)\in\SI_{n_2}^+,0\le t<\infty)$ satisfy the matrix Riccati equation,
\begin{eqnarray}
q_2^c(0)
  & = & e_2\etil_3^\prime \phi_{33} \etil_3e_2^\prime \nonumber \\
        [-1.5ex] \label{eq:criccati} \\ [-1.5ex]
\dot{q}_2^c
  & = & (b_{22}-a_{21}\gamma)q_2^c
        + q_2^c(b_{22}-a_{21}\gamma)^\prime  + g
        - q_2^c\gamma^\prime a_{11}\gamma q_2^c, \nonumber
\end{eqnarray}
where
\begin{equation}
g(t) = \left\{\begin{array}{ll}
              e_2\etil_3^\prime\alpha_{33}(t) \etil_3e_2^\prime & \text{if }0\le t\le s \\
              \alpha(t) & \text{otherwise}.
              \end{array} \right. \label{eq:gdef}
\end{equation}

\begin{proposition} \label{pr:transentc}
If (H1), (H2) and (H5) hold, then, for any $0\le s\le t<\infty$,
\begin{equation}
\Ttil_{3\rightarrow 1|2}^W(s,t)
  = \half\int_s^t\tr\left(\gamma(r)^\prime a_{11}(r)\gamma(r)
    (q_2^c(r)-\qtil_2^c(r))\right)dr, \label{eq:TW31val}
\end{equation}
where $(\qtil_2^c(t), s\le t<\infty)$ satisfies (\ref{eq:criccati}) over the time interval
$[s,\infty)$ and $\qtil_2^c(s)=0$.  (NB.~$\Ttil_{2\rightarrow 1}^W$ can be found from
(\ref{eq:T21val}), (\ref{eq:Wsplit}) and (\ref{eq:TW31val}).)
\end{proposition}

\begin{proof}
Fix $0\le s<\infty$, let $(\Phi(t)\in\R^n, 0\le t<\infty)$ satisfy
\[
\Phi(t) = \etil_2^\prime\Xitil_2 + \int_0^t b(r)\Phi(r) dr
          + \etil_2^\prime \Wtil_2(s\wedge t),
\]
and let $Z= X-\Phi$.  Then $\sigma(\Phi(0,s))=\cf_2^W(s)$, $\Phi(s,t)$ is
$\cf_2^W(s)$-measurable for all $t\ge s$, and $Z$ is independent of $\Phi$.  Now
\begin{eqnarray*}
\Ttil_{3\rightarrow 1|2}^W
  & = & I((\Xitil_3,\Wtil_3(0,s));\Xtil_1(s,t)\cond \Xtil_1(0,s),\Xitil_2,\Wtil_2(0,s)) \\
  & = & I((\Xitil_3,\Wtil_3(0,s));\Phi_1(s,t)+Z_1(s,t)\cond Z_1(0,s),\Phi(0,s)) \\
  & = & I((\Xitil_3,\Wtil_3(0,s));Z_1(s,t)\cond Z_1(0,s),\Phi(0,s)) \\
  & = & I((\Xi_2,W_2(0,s));Z_1(s,t)\cond Z_1(0,s))
        - I(\Phi(0,s);Z_1(s,t)\cond Z_1(0,s)) \\
  & = & I((\Xi_2,W_2(0,s));Z_1(s,t)\cond Z_1(0,s)),
\end{eqnarray*}
where we have used a change of variables argument in the integral of (\ref{eq:kldiv}) in
the third step, the chain rule in the fourth step, and the independence of $\Phi$ and
$Z$ in the final step.  The special case of Proposition \ref{pr:transent}, in which
$\phi_{22}=0$ and $\alpha_{22}(t)=0$ for $0\le t\le s$, shows that the
final mutual information here is equal to the right-hand side of (\ref{eq:TW31val}).
\end{proof}

We can clearly also use $\Ttil_{2\rightarrow 1}^W$ and $\Ttil_{3\rightarrow 1|2}^W$ to
define a pair of directed information quantities, $\Dtil_{2\rightarrow 1}^W$ and
$\Dtil_{3\rightarrow 1|2}^W$.


\begin{thebibliography}{10}

\bibitem{ammi1}
  P-O.~Amblard and O.J.J~Michel, The relation between Granger causality and directed
  information theory: a review, \textit{Entropy} \textbf{15} (2013) 113--143.
\bibitem{coth1}
  T.M.~Cover and J.A.~Thomas, \textit{Elements of Information Theory} (Wiley, 2006).
\bibitem{dunc1}
  T.E.~Duncan, On the calculation of mutual information, \textit{SIAM J.~Appl.~Math.}
  \textbf{19} (1970) 215--220.
\bibitem{kazz1}
  T.T.~Kadota, M.~Zakai and J.~Ziv, Mutual information of the white Gaussian channel
  with and without feedback, \textit{IEEE Trans. Information Theory} \textbf{17} (1971)
  368--371.
\bibitem{lish1}
  R.S.~Liptser and A.N.~Shiryayev, \textit{Statistics of Random Processes
  I---General Theory} (Springer, 2001).
\bibitem{mark1}
  H.~Marko, The bidirectional information theory---a generalisation of information theory,
  \textit{IEEE Trans.~Communications} \textbf{21} (1973) 1345--1351.
\bibitem{mass1}
  J.L.~Massey, Causality, feedback and directed information, in \textit{Proceedings of the
  International Syposium on Information Theory and its Applications}, (1990) 303--305.
\bibitem{mass2}
  J.L.~Massey and P.C.~Massey, Conservation of mutual and directed information, in
  \textit{Proceedings of the International Symposium on Information Theory}, (2005)
  157--158.
\bibitem{maza1}
  E.~Mayer-Wolf and M.~Zakai, (1984) On a formula relating the Shannon information
  to the Fisher information for the filtering problem, in: \textit{Filtering and
  Control of Random Processes}, H.~Korezlioglu, G.~Mazziotto and S.~Szpirglas, S. (eds.),
  \textit{Lecture Notes in Control and Information Sciences} \textbf{61} (Springer, 1984)
  164--171.
\bibitem{maza2}
  E.~Mayer-Wolf and M.~Zakai, Some relations between mutual information and estimation
  error in Wiener space, \textit{Ann.~Appl.~Probab.} \textbf{17} (2007) 1102--1116.
\bibitem{mine2}
  S.K.~Mitter and N.J.~Newton, Information and entropy flow in the Kalman-Bucy filter,
  \textit{J.~Statist.~Phys.} \textbf{118} (2005) 145--167.
\bibitem{newt1}
  N.J.~Newton, Dual Kalman-Bucy filters and interactive entropy production,
  \textit{SIAM J.~Control Optim.} \textbf{45} (2006) 998--1016.
\bibitem{newt2}
  N.J.~Newton, Dual nonlinear filters and entropy production,
  \textit{SIAM J.~Control Optim.} \textbf{46} (2007) 1637--1663.
\bibitem{newt3}
  N.J.~Newton, Interactive statistical mechanics and nonlinear filtering,
  \textit{J.~Statist~Phys.} \textbf{133} (2008) 711--737.
\bibitem{pear1}
  J.~Pearl, \textit{Causality: Models, Reasoning and Inference}, Cambridge University
  Press (2000).
\bibitem{pins1}
  M.S.~Pinsker, \textit{Information and Information Stability of Random Variables and
  Processes}, (Holden-Day, 1964)
\bibitem{rowi2}
   L.C.G.~Rogers and D.~Williams, \textit{Diffusions,
   Markov Processes and Martingales: Volume 1---Foundations}, (Cambridge University Press,
   2000)
\bibitem{sdnm1}
  H.~Sandberg, J-C.~Delvenne, N.J.~Newton and S.K.~Mitter, Maximum work extraction and
  implementation costs for nonequilibrium Maxwell's demons, \textit{Phys.~Rev.~E}
  \textbf{90} (2015) 042119
\bibitem{schr1}
  T.~Schreiber, Measuring information transfer, \textit{Phys.~Rev.~Lett.} \textbf{85}
  (2000) 461--465.
\bibitem{shan1}
  C.E.~Shannon, A mathematical theory of communication, \textit{Bell System Technical
  Journal}, \textbf{27} (1948) 379--423 and 623--656.
\bibitem{wekp1}
  T.~Weissman, Y-H.~Kim and H.H.~Permuter, Directed information, causal estimation and
  communication in continuous time, \textit{IEEE Trans.~Inform.~Theory} \textbf{59}
  (2013) 1271--1287.
\end{thebibliography}
\end{document}